\documentclass[12pt]{article}
\usepackage{latexsym, amsmath, amsfonts, amsthm, amssymb}
\usepackage{times}
\usepackage{a4wide}
\usepackage{graphicx, tocloft}
\usepackage{titlesec, titling}
\usepackage{tikz}
\usetikzlibrary{matrix,arrows}

\titleformat{\section}{\large\bfseries}{\thesection}{1em}{}
\titleformat{\subsection}{\normalsize\bfseries}{\thesubsection}{1em}{}

\def \proj {{\rm Proj}}
\def \effrank {{\rm eff.rank}}
\def \med {{\rm Median}}

\def \RR {\mathbb R}
\def \QQ {\mathbb Q}

\def \HH {\mathbb H}
\def \EE {\mathbb E}
\def \ZZ {\mathbb Z}

\def \PP {\mathbb P}

\def \eps {\varepsilon}

\def \cB {\mathcal B}
\def \cG {\mathcal G}
\def \cE {\mathcal E}

\def \cS {\mathcal S}

\def \cF {\mathcal F}

\def \cA {\mathcal A}

\def \cB {\mathcal B}

\def \cO {\mathcal O}

\newtheorem{theorem}{Theorem}[section]
\newtheorem{lemma}[theorem]{Lemma}
\newtheorem{proposition}[theorem]{Proposition}
\newtheorem{corollary}[theorem]{Corollary}
\newtheorem{remark}[theorem]{Remark}

\newtheorem{definition}[theorem]{Definition}

\def\myffrac#1#2 in #3{\raise 2.6pt\hbox{$#3 #1$}\mkern-1.5mu\raise 0.8pt\hbox{$
#3/$}\mkern-1.1mu\lower 1.5pt\hbox{$#3 #2$}}

\def\qed{\hfill $\vcenter{\hrule height .3mm
\hbox {\vrule width .3mm height 2.1mm \kern 2mm \vrule width .3mm
height 2.1mm} \hrule height .3mm}$ \bigskip}

\begin{document}

\title{Super-Gaussian directions of  random vectors}
\author{Bo'az Klartag}

\date{}
\maketitle

\begin{abstract} We establish the following universality property in high dimensions:
Let $X$ be a random vector with density in $\RR^n$. The density function can be arbitrary.
We show that there exists a fixed unit vector $\theta \in \RR^n$
such that the random variable $Y = \langle X, \theta \rangle$ satisfies
$$ \min \left \{ \PP( Y \geq t M ), \PP(Y \leq -tM) \right \} \geq c e^{-C t^2} \qquad \qquad \text{for all} \ 0 \leq t \leq \tilde{c} \sqrt{n}, $$
where $M > 0$ is any median of $|Y|$, i.e., $\min \{ \PP( |Y| \geq M), \PP( |Y| \leq M ) \} \geq 1/2$. Here, $c, \tilde{c}, C > 0$ are universal
constants. The dependence on the dimension $n$ is optimal, up to universal constants, improving upon our previous work.
\end{abstract}

\section{Introduction}

Consider a random vector $X$ that is distributed uniformly in some Euclidean ball centered at the origin in $\RR^n$.
For any fixed vector $0 \neq \theta \in \RR^n$, the density of the random variable $\langle X, \theta \rangle = \sum_i \theta_i X_i$
may be found explicitly, and in fact it is proportional to
the function
\begin{equation} 
 t \mapsto \left( 1 - \frac{t^2}{A^2 n} \right)_+^{(n-1)/2} \qquad \qquad (t \in \RR) \label{eq_1437} 
 \end{equation}
  where $x_+ = \max \{ x, 0 \}$ and $A > 0$ is a parameter depending on the length of $\theta$ and the radius of the Euclidean ball.
  It follows  that when the dimension $n$ is large, the density in (\ref{eq_1437}) is close to a Gaussian density, and   
  the random variable $Y = \langle X, \theta \rangle$ has a tail of considerable size:
\begin{equation}  \PP (Y \geq t M)  \geq c \exp(-C t^2) \qquad \qquad \text{for all} \ 0 \leq t \leq \tilde{c}  \sqrt{n}. \label{eq_1051_}
\end{equation}
Here, $M = \med(|Y|)$ is any median of $|Y|$, i.e., $\min \{ \PP( |Y| \geq M), \PP( |Y| \leq M ) \} \geq 1/2$,
and $c, \tilde{c}, C > 0$ are universal constants.
Both the median and the expectation of $|Y|$ differ from $A$ by a factor which is at most a universal constant.
We prefer to work with a median since in the cases we will consider shortly, the expectation of $|Y|$ is not guaranteed to be finite.
 The inequality in (\ref{eq_1051_}) expresses the property that the tail distribution
of $Y / M $ is at least as heavy as the standard Gaussian tail distribution, for $\sqrt{n}$ standard deviations.
The dependence on the dimension $n$ is optimal, since for $t > \tilde{C} \sqrt{n}$, the probability on the left-hand side of (\ref{eq_1051_}) vanishes.

\medskip
Our goal in this paper is to  show that a similar phenomenon occurs for essentially any
random vector in $\RR^n$, and not only for the uniform distribution on the high-dimensional Euclidean ball.
Recall that when $n$ is large and the random vector $X = (X_1,\ldots,X_n)$ has independent coordinates,
the classical central limit theorem implies that under mild assumptions, there exists $0 \neq \theta \in \RR^n$ for which $\langle X, \theta \rangle$ is approximately Gaussian.
It is curious to note that a Gaussian lower bound on the tail persists, even when the independence assumption is completely dropped.

\begin{figure}
\begin{center}
\begin{tikzpicture}[scale = 2]

\fill[fill=gray] (1.5,0) -- plot [domain=1.5:3] (\x,{ exp(-\x*\x/2) }) -- (3,0) -- cycle;

\fill[fill=gray, opacity=0.2] (-0.7,0) -- plot [domain=-0.7:0.7] (\x,{ exp(-\x*\x/2) }) -- (0.7,0) -- cycle;

\draw [<->] (-3.5,0) -- (3.5,0);
\draw [->] (0, -0.1) -- (0, 1.1);
\draw[samples=200, thick,  domain=-3:3] plot (\x, {   exp(-\x*\x/2)   });

\draw (0.7,-0.03) -- (0.7,0.03);
\draw (-0.7,-0.03) -- (-0.7,0.03);
\draw (1.5,-0.03) -- (1.5,0.03);

\node [below] at (1.5,0) { $t  M$};
\node [below] at (0.7,0) { $M$};
\node [below] at (-0.7,0) { $-M$};

\node [above] at (0,1.05) {\small {\rm area} $=1/2$};
\node [above] at (2.4,0.2) {\small {\rm area} $\geq c e^{-C t^2}$};

\end{tikzpicture}
\caption{An example of a density of a Super-Gaussian random variable \label{fig1}}
\end{center}
\end{figure}
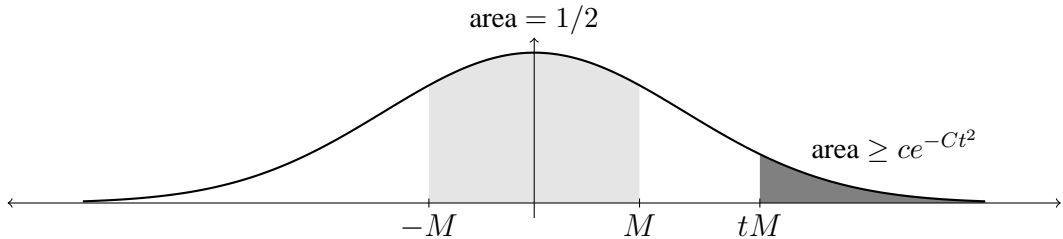

\medskip Let $Y$ be a real-valued random variable and let $L > 0$. We say that $Y$ is {\it Super-Gaussian of length $L$} with parameters $\alpha, \beta > 0$
if $\PP( Y = 0) = 0$  and for any $0 \leq t \leq L$,
$$ \min \left \{ \PP( Y \geq t M ), \PP(Y \leq -tM) \right \} \geq \alpha e^{-t^2/\beta}, $$
where $M = \med(|Y|)$ is any median of $|Y|$.
The requirement that $\PP( Y = 0) = 0$ is necessary only to avoid trivialities.
A Gaussian random variable is certainly super-Gaussian of infinite length, as well as a symmetric exponential random variable.
Write $|x| = \sqrt{\langle x, x \rangle}$ for the standard Euclidean norm of $x \in \RR^n$,
and denote  $S^{n-1} = \{ x \in \RR^n \, ;
\, |x| = 1 \}$.

\begin{theorem} Let $X$ be a random vector with density in $\RR^n$. Then there exists a fixed vector $\theta \in
S^{n-1}$
such that $\langle X, \theta \rangle$ is Super-Gaussian of length $c_1 \sqrt{n}$ with parameters $c_2, c_3 > 0$,
where $c_1,c_2,c_3 > 0$ are universal constants. \label{thm_1030}
\end{theorem}

Theorem \ref{thm_1030} improves upon Corollary 1.4 from \cite{K_euro}, in which the dependence on the dimension $n$
was logarithmic. In the case where $X$ is distributed uniformly in a $1$-unconditional convex body in $\RR^n$,
Theorem \ref{thm_1030} goes back to Pivovarov \cite{Pivo} up to logarithmic factors.
In the case where $X$ is distributed uniformly in a convex body satisfying the hyperplane conjecture with a uniform constant, Theorem
\ref{thm_1030} is due to Paouris \cite{Paouris}.
Theorem \ref{thm_1030} provides a universal lower bound on the tail distribution, which is tight up to constants in the case
where $X$ is uniformly distributed in a Euclidean ball centered at the origin. In particular,  the dependence on the dimension in Theorem \ref{thm_1030} is optimal, up to the value of the universal constants.

\medskip The assumption that the random vector $X$ has a density in $\RR^n$ may be somewhat  relaxed. The following definition appears in \cite{BLYZ1, K_euro} with minor modifications:

\begin{definition} Let $X$ be a random vector in a finite-dimensional vector space
$\cB$ and let $d > 0$. We say that ``the effective rank of $X$ is at least $d$'', or in short that $X$ is of class $\effrank_{\geq d}$
if for any
linear subspace $E \subseteq \cB$,
\begin{equation} \PP(X \in E) \leq \dim(E) / d, \label{eq_1011}
\end{equation}
with equality if and only if there is a subspace $F \subseteq B$ with $E \oplus F = \cB$ and $\PP(X \in E \cup F) = 1$.
\label{def_1131}
\end{definition}

Intuitively, when $X$ is of class $\effrank_{\geq d}$ we  think of the support of $X$ as effectively spanning a
subspace whose dimension is at least $d$. Note, however, that $d$ is not necessarily an integer.
By substituting $E = \cB$ in (\ref{eq_1011}), we see that there are no random
vectors in $\RR^n$ of class $\effrank_{\geq d}$ with $d > n$. We say that the effective rank
of $X$ is $d$ when  $X$ is of class $\effrank_{\geq d}$, but for any $\eps > 0$ the random vector $X$ is not of class $\effrank_{\geq d + \eps}$.
The effective rank of $X$ is $d^-$ if $X$ is of class $\effrank_{\geq d-\eps}$ for all $0 < \eps < d$ but $X$ is not of class $\effrank_{\geq d}$.
In the terminology of \cite{K_euro}, the random vector $X$ has an effective rank greater than $d$ if and only if it is $\eps$-decent
for some $\eps < 1/d$.

\medskip There are many random vectors in $\RR^n$ whose effective rank is precisely $n$.
For example, any random vector with density in $\RR^n$, or any random vector $X$ that is distributed
uniformly on a finite set that spans $\RR^n$ and does not contain the origin.
It was shown by B\"or\"oczky, Lutwak, Yang, and Zhang
\cite{BLYZ2} and by Henk and Linke \cite{HL} that the cone volume measure of any convex body in $\RR^n$ with barycenter at the origin
is of class $\effrank_{\geq n}$ as well.
Note that a random variable $Y$ is Super-Gaussian of length $L$ with parameters $\alpha, \beta > 0$ if and only
if for any number $0 \neq r \in \RR$, also $r Y$ is Super-Gaussian of length $L$ with the same parameters $\alpha, \beta > 0$.
Theorem \ref{thm_1030} is thus a particular case of the following:

\begin{theorem} Let $d \geq 1 $ and let $\cB$ be a finite-dimensional linear space. Let $X$ be a random vector in $\cB$ whose effective rank is at least $d$. Then there exists a non-zero, fixed,
linear functional $\ell: \cB \rightarrow \RR$ such that the random variable $\ell(X)$ is Super-Gaussian of length $c_1 \sqrt{d}$ with parameters $c_2, c_3 > 0$,
where $c_1,c_2,c_3 > 0$ are universal constants.  \label{thm_1140}
\end{theorem}

Theorem \ref{thm_1140}
admits the following corollary, pertaining to infinite-dimensional spaces:

\begin{corollary} Let $\cB$ be a topological vector space with a countable family of continuous linear functionals that separates points in $\cB$.
Let $X$ be a random vector, distributed according to a Borel probability measure in $\cB$. Assume that $d \geq 1$ is such that $\PP(X \in E) \leq \dim(E)  / d$
for any finite-dimensional subspace $E \subseteq \cB$.

\medskip Then there exists a non-zero, fixed, continuous linear
functional $\ell: \cB \rightarrow \RR$ such that the random variable $\ell(X)$ is Super-Gaussian of length $c_1 \sqrt{d}$ with parameters $c_2, c_3 > 0$,
where $c_1,c_2,c_3 > 0$ are universal constants. \label{cor_1405}
\end{corollary}

The remainder of this paper is devoted to the proof of Theorem \ref{thm_1140} and Corollary \ref{cor_1405}.
We use the letters
$c, C, \tilde{C}, c_1, C_2$ etc. to denote various positive universal constants, whose value may change from one line to the next.
We  use upper-case $C$ to denote universal constants that we think of as ``sufficiently large'', and lower-case
$c$ to denote universal constants that are ``sufficiently small''.
We write $\#(A)$ for the cardinality of a set $A$.
When we write that a certain set or a certain number
are fixed, we intend to emphasize that they are non-random.

\medskip We denote by $\sigma_{n-1}$  the uniform probability measure on the sphere $S^{n-1}$,
which is the unique rotationally-invariant probability measure on $S^{n-1}$.
When we say that a random vector $\theta$ is distributed uniformly on $S^{n-1}$, we refer
to the probability measure $\sigma_{n-1}$. Similarly, when we write that a random subspace $E$
is distributed uniformly over the Grassmannian $G_{n,k}$ of $k$-dimensional subspaces of $\RR^n$,
we refer to the unique rotationally-invariant probability measure on $G_{n,k}$.

\medskip
{\emph{Acknowledgements}.} I would like to thank Bo Berndtsson and Emanuel Milman for
interesting discussions and for encouraging me to write this paper.
Supported by a grant from the European Research Council.

\section{Proof strategy}
\label{sec_stra}

The main ingredient in the proof of Theorem \ref{thm_1140} is the following proposition:

\begin{proposition} Let $X$ be a random vector in $\RR^n$ with $\PP(X = 0) = 0$ such that
 \begin{equation} \EE \left \langle \frac{X}{|X|}, \theta \right \rangle^2  \leq \frac{5}{n} \qquad \qquad \text{for all} \ \theta \in S^{n-1}. \label{eq_1051}
 \end{equation}
Then there exists a fixed vector $\theta \in S^{n-1}$ such that the random variable $\langle X, \theta \rangle$ is Super-Gaussian of length $c_1 \sqrt{n}$ with parameters $c_2, c_3 > 0$,
where $c_1,c_2,c_3 > 0$ are universal constants.
\label{prop_818}
\end{proposition}

The number $5$ in Proposition \ref{prop_818} does not play any particular role, and  may be replaced by any other
universal constant, at the expense of modifying the values of $c_1, c_2$ and $c_3$. Let us explain the key ideas in the proof of Proposition \ref{prop_818}.
 In our previous work \cite{K_euro}, the unit vector $\theta \in S^{n-1}$
was chosen randomly, uniformly on $S^{n-1}$. In order to improve the dependence on
the dimension, here we select $\theta$ a bit differently. We shall define $\theta_1$ and $\theta_2$ via the following procedure:
\begin{enumerate}
\item[(i)] Let $M > 0$ be a $1/3$-quantile of $|X|$, i.e., $\PP(|X| \geq M) \geq 1/3$ and $\PP(|X| \leq M) \geq 2/3$. We fix a vector $\theta_1 \in S^{n-1}$ such that
$$ \PP \left( |X| \geq M \ \text{and} \ \left| \frac{X}{|X|} - \theta_1 \right| \leq \frac{1}{5}  \right) \geq \frac{1}{2} \cdot
\sup_{\eta \in S^{n-1}} \PP \left( |X| \geq M \ \text{and} \ \left| \frac{X}{|X|} - \eta \right| \leq \frac{1}{5}  \right). $$
\item[(ii)] Next, we fix a vector $\theta_2 \in S^{n-1}$ with $|\langle \theta_1, \theta_2  \rangle| \leq 1/10$ such that
$$ \PP \left( |X| \geq M \ \text{and} \ \left| \frac{X}{|X|} - \theta_2 \right| \leq \frac{1}{5}  \right) \geq \frac{1}{2} \cdot
\sup_{\eta \in S^{n-1} \atop{|\langle \eta, \theta_1 \rangle| \leq 1/10}} \PP \left( |X| \geq M \ \text{and} \ \left| \frac{X}{|X|} - \eta \right| \leq \frac{1}{5}  \right). $$
\end{enumerate}
In the following pages we will   describe a certain subset $\cF_3 \subseteq S^{n-1}$ which satisfies $\sigma_{n-1}(\cF_3) \geq 1 - C / n^c$ and
$\theta_2 - \theta_1 \not \in \cF_3$.
We will show that for any $\theta_3 \in \cF_3$,  the random variable
$ \langle X, \theta \rangle$ is Super-Gaussian of length $c \sqrt{n}$ with parameters $c_1,c_2 > 0$,
where $\theta$ is defined as follows:
\begin{equation}  \theta = \frac{\theta_1 - \theta_2 + \theta_3}{|\theta_1 - \theta_2 + \theta_3|}. \label{eq_741} \end{equation}
Thus, $\theta_1$ and $\theta_2$ are fixed vectors, while most choices of $\theta_3$ will work for us, where by
``most'' we refer to the uniform measure on $S^{n-1}$. The first step the proof below
is to show that for any unit vector $\theta \in S^{n-1}$,
\begin{equation}
\med \left( |\langle X, \theta \rangle| \right) \leq C M / \sqrt{n}, \label{eq_1649}
\end{equation}
that is, any median of $ |\langle X, \theta \rangle|$ is at most $C M / \sqrt{n}$.
Then we need to show that when $\theta_3 \in \cF_3$ and $\theta$ is defined as in (\ref{eq_741}), for all $0 \leq t \leq c \sqrt{n}$,
\begin{equation} \min \left \{ \PP \left( Y \geq \frac{t M}{\sqrt{n}} \right), \PP \left(Y \leq - \frac{tM}{\sqrt{n}} \right) \right \} \geq \tilde{c} e^{-\tilde{C} t^2}. \label{eq_1654}
\end{equation}
The proof of (\ref{eq_1654}) is divided into three sections.
The case where $t \in [0, \sqrt{\log n}]$ may essentially be handled
by using the methods of \cite{K_euro}, see Section \ref{sec2}. Let $t_0 > 0$ be defined via
\begin{equation}
e^{-t_0^2} = \PP \left( |X| \geq M \ \text{and} \ \left| \frac{X}{|X|} - \theta_2 \right| \leq \frac{1}{5}  \right). \label{eq_954}
\end{equation}
In order to prove (\ref{eq_1654}) in the range $t \in [\sqrt{\log n}, t_0]$, we will
use tools from the local theory of Banach spaces, such as Sudakov's inequality as well as the concentration of measure on
the sphere. Details in Section \ref{sec3} below.
The remaining interval $t \in [t_0, c \sqrt{n}]$ is analyzed in Section \ref{sec4}. In Section \ref{sec_ang} we deduce
Theorem \ref{thm_1140} and Corollary \ref{cor_1405} from Proposition \ref{prop_818} by using the angularly-isotropic position, along the lines of \cite{K_euro}.

\section{Central limit regime}
\label{sec2}

This section is the first in a sequence of three sections that are dedicated to the proof
of Proposition \ref{prop_818}. Thus, we are given a random vector $X$ in $\RR^n$ with $\PP(X = 0) = 0$ such that
(\ref{eq_1051}) holds true.  We fix a number $M > 0$ with the property that
\begin{equation}
\PP(|X| \geq M) \geq 1/3, \qquad \PP(|X| \leq M) \geq 2/3.
\label{eq_1630}
\end{equation}
That is, $M$ is a $1/3$-quantile of $|X|$.
Our first lemma
verifies (\ref{eq_1649}), as it
states that for any choice of a unit vector $\theta$, any median of the random
variable $|\langle X, \theta \rangle|$ is at most $C M / \sqrt{n}$.

\begin{lemma} For any $\theta \in S^{n-1}$,
$$ \PP \left( |\langle X, \theta \rangle| \geq C M / \sqrt{n} \right) < 1/2, $$
where $C > 0$ is a universal constant.
\label{lem_746}
\end{lemma}

\begin{proof} It follows from (\ref{eq_1051}) that for any $\theta \in S^{n-1}$,
$$  \EE \left[ \left \langle X, \theta \right \rangle^2 1_{\{ |X| \leq M \}} \right]
\leq  \EE  \left[ \left \langle X, \theta \right \rangle^2 \cdot \frac{M^2}{|X|^2} \right]
= M^2 \cdot \EE \left \langle \frac{X}{| X |}, \theta \right \rangle^2 \leq \frac{5 M^2}{n}. $$
By the Markov-Chebyshev inequality,
$$ \PP \left( \left \langle X, \theta \right \rangle^2 1_{\{ |X| \leq M \}} \geq 35 M^2 / n \right) \leq 1/7. $$
Since $\PP(|X| > M) \leq 1/3$, we obtain
$$ \PP \left( |\langle X, \theta \rangle| \geq \frac{6 M}{\sqrt{n}} \right) \leq \PP( |X| > M ) + \PP \left( |\langle X, \theta \rangle| \geq \frac{6 M}{\sqrt{n}} \ \text{and} \ |X| \leq M \right)
\leq \frac{1}{3} + \frac{1}{7} < \frac{1}{2}. $$
The lemma follows with $C = 6$.
\end{proof}

The rest of this section is devoted to the proof of (\ref{eq_1654}) in the range $t \in [0, \sqrt{\log n}]$.
The defining properties of $\theta_1, \theta_2 \in S^{n-1}$ from the previous section will not be used here,
the entire analysis in this section applies for arbitrary unit vectors $\theta_1$ and $\theta_2$.

\begin{lemma} Let $\theta_1,\theta_2 \in S^{n-1}$ be any two fixed vectors. Then,
$$ \PP \left(
|X| \geq M, \ |\langle X, \theta_1 \rangle| \leq \frac{10 |X|}{\sqrt{n}} \ \ \text{and} \ \   |\langle X, \theta_2 \rangle| \leq \frac{10 |X|}{\sqrt{n}} \right)
 > \frac{1}{5}. $$ \label{lem_519}
\end{lemma}

\begin{proof} By (\ref{eq_1051}) and the Markov-Chebyshev inequality, for $j=1,2$,
$$ \PP \left( |\langle X, \theta_j \rangle| \geq \frac{10 |X|}{\sqrt{n}} \right) \leq \frac{n}{100} \cdot \EE \left \langle \frac{X}{|X|}, \theta_j \right \rangle^2
\leq \frac{n}{100} \cdot \frac{5}{n} = \frac{1}{20}. $$
Thanks to (\ref{eq_1630}), we conclude that
\begin{equation*}  \PP \left(
|X| \geq M, \ |\langle X, \theta_1 \rangle| \leq \frac{10 |X|}{\sqrt{n}}, \  |\langle X, \theta_2 \rangle| \leq \frac{10 |X|}{\sqrt{n}} \right)
\geq 1 - \left( \frac{2}{3} + \frac{1}{20} + \frac{1}{20} \right) > \frac{1}{5}. \tag*{\qedhere} \end{equation*}
\end{proof}

Let $1 \leq k \leq n$. Following \cite{K_euro}, we write $\cO_k \subseteq (\RR^n)^k$ for the collection of all $k$-tuples $(v_1,\ldots,v_k)$ with the following property: There exist orthonormal vectors $w_1,\ldots,w_{k} \in \RR^n$
and real numbers $(a_{ij})_{i,j=0,\ldots,k}$ such that $|a_{ij}| < a_{ii} / k^2$ for $j < i$, and
\begin{equation} v_i = \sum_{j=1}^i a_{ij} w_j \ \ \ \ \ \ \text{for} \ i=1,\ldots,k. \label{eq_1014}
\end{equation}
In other words, $\cO_k$ consists of $k$-tuples of vectors that are almost orthogonal.
By recalling the Gram-Schmidt process from linear algebra, we see that $(v_1,\ldots,v_k) \in \cO_k$ assuming that
\begin{equation}  |\proj_{E_{i-1}} v_i| < |v_i| / k^2 \qquad \text{for} \ i=1,\ldots,k, \label{eq_1634} \end{equation}
where $E_{i}$ is the subspace spanned by the vectors $v_1,\ldots,v_i \in \RR^n$ and $\proj_{E_i}$ is the orthogonal projection operator onto $E_i$ in $\RR^n$.
Here, $E_0 = \{ 0 \}$.

\begin{lemma} Assume that $1 \leq k \leq n$ and fix $(v_1,\ldots,v_k) \in \cO_k$. Then there exists $\cF \subseteq S^{n-1}$ with $\sigma_{n-1}(\cF) \geq 1 - C \exp(-c \sqrt{k})$
such that for any $\theta \in \cF$ and  $0 \leq t \leq \sqrt{\log k}$,
$$ \# \left \{ 1 \leq i \leq k \, ; \, \langle v_i, \theta \rangle \geq c_1 \frac{ |v_i|}{\sqrt{n}} \cdot t \right \} \geq c_2 e^{-C_3 t^2} \cdot k, $$
where $c_1, c_2, C_3, c, C > 0$ are universal constants. \label{lem_1013}
\end{lemma}

\begin{proof} Let $w_1,\ldots, w_k$ and $(a_{ij})$ be as in  (\ref{eq_1014}).
By applying an orthogonal transformation in $\RR^n$, we may assume that $w_i = e_i$, the standard $i^{th}$ unit vector.
Let $\Gamma = (\Gamma_1,\ldots,\Gamma_n) \in \RR^n$ be a standard Gaussian random vector in $\RR^n$.
For  $i=1,\ldots,n$ and $t > 0$, it is well-known that
$$ \PP(\Gamma_i \geq t) = \frac{1}{\sqrt{2 \pi}} \int_t^{\infty} e^{-s^2/2} ds \in [c  e^{-t^2}, C e^{-t^2/2} ].
$$
Therefore, by the Chernoff large deviations bound (e.g., \cite[Chapter 2]{BMS}), for any $t > 0$,
\begin{equation} \PP \left( \# \left \{ 1 \leq i \leq k \, ; \, \Gamma_i \geq t \right \} \geq
  \frac{c}{2} \cdot e^{-t^2} \cdot k \right) \geq 1 - \tilde{C} \exp \left(- \tilde{c} e^{-t^2} k  \right). \label{eq-1053} \end{equation}
From the Bernstein large deviation inequality (e.g., \cite[Chapter 2]{BMS}),
\begin{equation} \PP \left( |\Gamma| \leq 2 \sqrt{n} \right) \geq 1 - C e^{-cn},
\qquad \PP \left( \sum_{i=1}^k |\Gamma_i| \leq 2 k \right) \geq 1 - \hat{C} e^{-\hat{c} k}.
\label{eq_1055} \end{equation}
Note that when $\sum_{i=1}^k |\Gamma_i| \leq 2k$, for any $i=1,\ldots,k$,
\begin{equation}  \langle \Gamma, v_i \rangle = a_{ii} \cdot \left \langle \Gamma, e_i + \sum_{j=2}^i \frac{a_{ij}}{a_{ii}} e_j \right \rangle
\geq  a_{ii} \left( \Gamma_i - \frac{\sum_{j=1}^k |\Gamma_j|}{k^2} \right) \geq a_{ii} \left( \Gamma_i - \frac{2}{k} \right).
\label{eq_1053} \end{equation}
Moreover, $a_{ii} = |v_i - \sum_{j \leq 2} a_{ij} e_j| \geq |v_i| - a_{ii} / k$ for all $i=1,\ldots,k$. Therefore $a_{ii} \geq |v_i| / 2$ for all $i$.
It thus follows from (\ref{eq_1053}) that when $\sum_{i=1}^k |\Gamma_i| \leq 2k$, for any $i$,
$$ \Gamma_i \geq t \qquad \Longrightarrow \quad \langle \Gamma, v_i \rangle \geq a_{ii} \Gamma_i / 2 \geq |v_i| t / 4
\qquad \qquad \qquad \text{for all} \ t \geq 4 / k. $$
Hence we deduce from (\ref{eq-1053}) and (\ref{eq_1055}) that for all $t \geq 4 / k$,
\begin{equation} \PP \left( \# \left \{ i  \, ; \, \langle \Gamma, v_i \rangle \geq \frac{ t |v_i|}{4} \right \} \geq  \frac{c}{2} \cdot e^{-t^2} \cdot k \right) \geq 1 - \tilde{C} \exp \left(- \tilde{c} e^{-t^2} k  \right). \label{eq_1100} \end{equation}
Write $I = \{ \ell \in \ZZ \, ; \, \ell \geq 2, \, 2^{\ell} \leq \sqrt{\log k} / 5 \}$. By
substituting $t = 2^{\ell}$ into (\ref{eq_1100}) we see that
$$ \PP \left( \forall \ell \in I, \# \left \{ i  \, ; \, \langle \Gamma, v_i \rangle \geq  2^{\ell-2} |v_i| \right \} \geq  \frac{c}{2} \cdot e^{-(2^{\ell})^2} \cdot k \right) \geq 1 - \tilde{C} \sum_{\ell \in I} \exp \left(- \tilde{c} e^{-(2^{\ell})^2} k  \right).
$$
The latter sum is at most $\overline{C} \exp(-\overline{c} \sqrt{k})$. Moreover, suppose that $x \in \RR^n$ is a fixed vector such that $\# \left \{ i  \, ; \, \langle x, v_i \rangle \geq t |v_i|/4 \right \} \geq (c/2) e^{-t^2} k$ for all $1 \leq t \leq \sqrt{\log k} / 5$
of the form $t = 2^{\ell}$ for an integer $\ell \geq 2$. By adjusting the constants, we see that for any real number $t$ with $0 \leq t \leq \sqrt{\log k}$,
$$ \# \left \{ i  \, ; \, \langle x, v_i \rangle \geq c_1 t |v_i| \right \} \geq \tilde{c} e^{-\tilde{C} t^2} k. $$
Consequently,
$$ \PP \left( \forall t \in [0, \sqrt{\log k}], \ \# \left \{ i  \, ; \, \langle \Gamma, v_i \rangle \geq  c_1 t |v_i| \right \} \geq  \tilde{c} e^{-\tilde{C} t^2} \cdot k \right) \geq 1 - \overline{C} e^{-\overline{c} \sqrt{k}}.
$$
Recall that $|\Gamma| \leq 2 \sqrt{n}$ with a probability of at least $1 - C e^{-c n}$. Therefore, as $k \leq n$,
\begin{equation} \PP \left( \forall t \in [0, \sqrt{\log k}], \ \# \left \{ i  \, ; \, \left \langle \frac{\Gamma}{|\Gamma|}, v_i \right \rangle \geq  c_1 \frac{t |v_i|}{2 \sqrt{n}} \right \} \geq  \tilde{c} e^{-\tilde{C} t^2} \cdot k \right) \geq 1 - \hat{C} e^{-\hat{c} \sqrt{k}}.
\label{eq_1127} \end{equation}
Since  $\Gamma / |\Gamma|$ is distributed uniformly on $S^{n-1}$, the lemma follows from (\ref{eq_1127}).
\end{proof}

Let $E \subseteq \RR^n$ be an arbitrary subspace. It follows from (\ref{eq_1051}) that
\begin{equation}
 \EE \left| \proj_{E} \frac{X}{|X|} \right|^2 = \EE \sum_{i=1}^{\dim(E)} \left \langle \frac{X}{|X|}, u_i \right \rangle^2
  \leq 5 \frac{\dim(E)}{n}, \label{eq_1117} \end{equation}
 where $u_1,\ldots,u_m$ is an orthonormal basis
 of the subspace $E$ for $m = \dim(E)$.

\begin{lemma} Set $\ell = \lfloor n^{1/8} \rfloor$ and let $\theta_1,\theta_2 \in S^{n-1}$ be any fixed vectors.
Let $X_1,\ldots,X_\ell$ be independent copies of the random vector $X$. Then with a probability of at least $1 - C / \ell$
of selecting $X_1,\ldots,X_{\ell}$, there exists a subset $I \subseteq \{ 1,\ldots,\ell \}$ with the following three properties:
\begin{enumerate}
\item[(i)] $\displaystyle k : = \#(I) \geq \ell / 10$.
\item[(ii)] We may write $I = \{ i_1,\ldots,i_{k} \}$ such that $(X_{i_1},\ldots,X_{i_{k}}) \in \cO_k$.
\item[(iii)] For  $j=1,\ldots,k$, $$ |X_{i_j}| \geq M, \quad |\langle X_{i_j}, \theta_1 \rangle| \leq 10 |X_{i_j}| /\sqrt{n} \quad \text{and}
\quad |\langle X_{i_j}, \theta_2 \rangle| \leq 10 |X_{i_j}|/\sqrt{n}. $$
\end{enumerate}
Here, $C > 0$ is a universal constant.
\label{lem_1009}
\end{lemma}

\begin{proof} We may assume that $\ell \geq 10$, as otherwise the lemma trivially holds with any $C \geq 10$.
Define
$$  I = \left \{ 1 \leq i \leq \ell \, ; \, |X_i| \geq M, \, |\langle X_{i}, \theta_1 \rangle| \leq 10 |X_i| /\sqrt{n},
\, |\langle X_{i}, \theta_2 \rangle| \leq 10 |X_i| /\sqrt{n} \right \}. $$
Denote $k = \#(I)$ and let $i_1 < i_2 < \ldots < i_{k}$ be the elements of $I$.
We conclude from Lemma \ref{lem_519} and the Chernoff large deviation bound that
\begin{equation}  \PP ( \#(I) \geq \ell/10 ) \geq 1 - C \exp(-c \ell). \label{eq_1617} \end{equation}
Thus (i) holds with a probability of at least $1 - C \exp(-c \ell)$. Clearly (iii) holds true with probability one, by the definition of $I$.
All that remains is to show that (ii) holds true with a probability of at least $1 - 1 / \ell$.
Write $F_i$ for the subspace
spanned by $X_1,\ldots,X_i$, with $F_0 = \{ 0 \}$. It follows from (\ref{eq_1117}) that
for $i=1,\ldots,\ell$,
$$ \EE \left| Proj_{F_{i-1}} \frac{X_i}{|X_i|} \right|^2 \leq \frac{5 \cdot \dim(F_{i-1})}{n} \leq \frac{5(i-1)}{n} \leq \frac{5\ell}{n} < \frac{1}{\ell^6},  $$
as $10 \leq \ell \leq n^{1/8}$. It follows from the Markov-Chebyshev inequality that with a probability of at least $1 - 1/\ell$,
$$ \left| Proj_{F_{i-1}} \frac{X_i}{|X_i|} \right| < \frac{1}{\ell^2} \qquad \qquad \text{for all} \ i=1,\ldots,\ell. $$
Write $E_j$ for the subspace spanned by $X_{i_1},\ldots,X_{i_j}$. Then $E_{j-1} \subseteq F_{i_j-1}$. Therefore,
with a probability of at least $1 - 1/\ell$,
$$ \left| Proj_{E_{j-1}} \frac{X_{i_j}}{|X_{i_j}|} \right| \leq \left| \proj_{F_{i_j-1}}  \frac{X_{i_j}}{|X_{i_j}|}  \right| < \frac{1}{\ell^2} \leq \frac{1}{k^2} \qquad \qquad \text{for all} \ j=1,\ldots,k. $$
In view of (\ref{eq_1634}), we see that (ii) holds true with a probability of at least $1 - 1/\ell$, thus completing the proof of the lemma.
\end{proof}

By combining Lemma \ref{lem_1013} and Lemma \ref{lem_1009} we arrive at the following:

\begin{lemma} Let $\ell, \theta_1, \theta_2$ be as in Lemma \ref{lem_1009}.
Then there exists a fixed subset $\cF \subseteq S^{n-1}$ with $\sigma_{n-1}(\cF) \geq 1 - C / \sqrt{\ell}$
such that for any $\theta_3 \in \cF$ the following holds: Define $\theta$ via (\ref{eq_741}).
Let $X_1,\ldots,X_\ell$ be independent copies of the random vector $X$. Then with a probability of at least $1 - C / \sqrt{\ell}$
of selecting $X_1,\ldots,X_{\ell}$,
\begin{equation}  \# \left \{ 1 \leq i \leq \ell \, ; \, \langle X_i, \theta \rangle \geq c_1 \frac{ M}{\sqrt{n}} \cdot t \right \} \geq c_2 e^{-C_3 t^2} \cdot \ell,
\qquad \qquad \text{for all} \ 0 \leq t \leq \sqrt{\log \ell}, \label{eq_1150} \end{equation}
and
\begin{equation} \# \left \{ 1 \leq i \leq \ell \, ; \, \langle X_i, \theta \rangle \leq -c_1 \frac{ M}{\sqrt{n}} \cdot t \right \} \geq c_2 e^{-C_3 t^2} \cdot \ell,
\qquad \qquad \text{for all} \ 0 \leq t \leq \sqrt{\log \ell}. \label{eq_1151} \end{equation}
Here, $c_1, c_2, C_3, c, C > 0$ are universal constants.
\label{lem_1114}
\end{lemma}

\begin{proof} Let $\Theta$ be a random vector, distributed uniformly on $S^{n-1}$.
According to Lemma \ref{lem_1009}, with a probability of at least $1 - C / \ell$ of selecting $X_1,\ldots,X_\ell$,
there exists a subset $$ I = \{ i_1,\ldots,i_k \} \subseteq \{ 1,\ldots,\ell \} $$
such that properties (i), (ii) and (iii)
of Lemma \ref{lem_1009} hold true. Let us apply Lemma \ref{lem_1013}.
Then under the event where
properties (i), (ii) and (iii) hold true,
with a probability of at least $1 - \tilde{C} \exp(-\tilde{c} \sqrt{\ell})$ of selecting $\Theta \in S^{n-1}$,
$$ \# \left \{ 1 \leq j \leq k \, ; \, \langle X_{i_j}, \Theta \rangle \geq c_1 \frac{ |X_{i_j}|}{\sqrt{n}} \cdot t \right \} \geq c_2 e^{-C_3 t^2} \cdot k
\qquad \qquad \text{for all} \ 0 \leq t \leq \sqrt{\log k}, $$
and moreover $k \geq \ell / 10$ with
$$  \max \left \{ \left|\left \langle \frac{X_{i_j}}{|X_{i_j}|}, \theta_1 \right \rangle \right|, \left|\left \langle \frac{X_{i_j}}{|X_{i_j}|}, \theta_2 \right \rangle \right| \right \}
\leq \frac{10}{\sqrt{n}} \qquad \qquad \text{for} \ j=1,\ldots,k. $$
Consequently, under the event where
properties (i), (ii) and (iii) hold true, with a probability of at least $1 - \tilde{C} \exp(-\tilde{c} \sqrt{\ell})$ of selecting $\Theta \in S^{n-1}$,
$$ \# \left \{ 1 \leq j \leq k \, ; \, \left \langle \frac{ X_{i_j} }{|X_{i_j}|}, \theta_1 - \theta_2 + \Theta
\right \rangle \geq \frac{c_1}{2} \frac{ t}{\sqrt{n}} \right \} \geq c_2 e^{-C_3 t^2} \cdot k
 \qquad \text{for} \ t \in [80/c_1,\sqrt{\log k}]. $$
Since $k \geq \ell/10$, the condition $t \in [80/c_1,\sqrt{\log k}]$ can be upgraded to $t \in [0, \sqrt{\log \ell}]$ at the cost of modifying the universal constants.
 Recall that by Lemma \ref{lem_1013}(iii), we have that $|X_{i_j}| \geq M$ for all $j$.
 By the triangle inequality, with probability one, $0 < |\theta_1 - \theta_2 + \Theta| \leq 3$. Hence,
 $$ |X_{i_j}| / |\theta_1 - \theta_2 + \Theta|  \geq M / 3. $$
Therefore, under the event where
properties (i), (ii) and (iii) hold true, with a probability of at least $1 - \tilde{C} \exp(-\tilde{c} \sqrt{\ell})$ of selecting $\Theta \in S^{n-1}$,
\begin{equation} \forall t \in [0, \sqrt{\log \ell}], \quad
\# \left \{ 1 \leq i \leq \ell \, ; \, \left \langle X_{i}, \frac{\theta_1 - \theta_2 + \Theta}{|\theta_1 - \theta_2 + \Theta|} \right \rangle \geq \bar{c}_1 \frac{ M}{\sqrt{n}} \cdot t \right \} \geq \bar{c}_2 e^{-\bar{C}_3 t^2} \cdot \ell. \label{eq_1044} \end{equation}
Write $\cA$ for the event that the statement in (\ref{eq_1044}) holds true. Denoting $\vec{X} = (X_1,\ldots,X_{\ell})$, we have shown that
$$ \PP( (\Theta, \vec{X}) \in \cA ) \geq 1 - \tilde{C} \exp(-\tilde{c} \sqrt{\ell}) - C / \ell \geq 1 - \bar{C} / \ell. $$ Denote
$$ \cF = \left \{ \theta \in S^{n-1} \, ; \, \PP_{\vec{X}}( (\theta, \vec{X}) \in \cA ) \geq 1 - \bar{C}/\sqrt{\ell} \right \}. $$
Then,
\begin{equation} \label{eq_1101} 1 - \frac{\bar{C}}{\ell} \leq \PP( (\Theta, \vec{X}) \in \cA ) \leq \PP( \Theta \in \cF) + \left( 1 - \frac{\bar{C}}{\sqrt{\ell}} \right) \PP( \Theta \not \in \cF).  \end{equation}
It follows from (\ref{eq_1101}) that $\sigma_{n-1}(\cF) = \PP( \Theta \in \cF) \geq 1 - 1 / \sqrt{\ell}$.
By the definition of $\cF \subseteq S^{n-1}$, for any $\theta_3 \in \cF$,
with a probability of at least $1 - \bar{C}{\sqrt{\ell}}$ of selecting $X_1,\ldots,X_{\ell}$,
$$ \forall t \in [0, \sqrt{\log \ell}], \quad
\# \left \{ 1 \leq i \leq \ell \, ; \, \left \langle X_{i}, \frac{\theta_1 - \theta_2 + \theta_3}{|\theta_1 - \theta_2 + \theta_3|} \right \rangle \geq \bar{c}_1 \frac{ M}{\sqrt{n}} \cdot t \right \} \geq \bar{c}_2 e^{-\bar{C}_3 t^2} \cdot \ell. $$
This completes the proof of (\ref{eq_1150}). The argument for (\ref{eq_1151}) requires only the most trivial modifications, and we leave it for the reader to complete.
\end{proof}

We will use the well-known fact that for any random variable $Y$ and measurable sets $A_1,\ldots,A_\ell$, by the Markov-Chebyshev inequality,
$$ \frac{1}{s} \cdot \sum_{i=1}^\ell \PP(Y \in A_i) = \frac{1}{s} \cdot \EE \sum_{i=1}^\ell 1_{ \{ Y \in A_i \} }
 \geq \PP \big( \# \left \{ i \, ; \, Y \in A_i \right \}  \geq s \big)
\qquad \qquad
\qquad (s > 0). $$

\begin{corollary}
Let $\theta_1,\theta_2 \in S^{n-1}$ be any fixed vectors.
Then there exists a fixed subset $\cF \subseteq S^{n-1}$ with $\sigma_{n-1}(\cF) \geq 1 - C / n^{c}$
such that for any $\theta_3 \in \cF$, defining $\theta$ via (\ref{eq_741}),
$$ \forall t \in [0, 5 \sqrt{\log n}], \quad
\min \left \{ \PP \left( \langle X, \theta \rangle \geq c_1 \frac{M}{\sqrt{n}} \cdot t \right),
\PP \left( \langle X, \theta \rangle \leq -c_1 \frac{M}{\sqrt{n}} \cdot t \right) \right \}  \geq c_2 e^{-C_3 t^2}, $$
where $c, C, c_1, c_2, C_3 > 0$ are universal constants.
\label{cor_103}
\end{corollary}

\begin{proof} We may assume that $n$ exceeds a certain fixed universal constant, as otherwise the conclusion of the lemma
trivially holds for $\cF = \emptyset$. Set $\ell = \lfloor n^{1/8} \rfloor$ and let $\cF$ be the set from Lemma \ref{lem_1114}.
Let $\theta_3 \in \cF$ and define $\theta$ via (\ref{eq_741}).
Suppose that $X_1,\ldots,X_\ell$ are independent copies of the random vector $X$. Then  for any $0 \leq t \leq \sqrt{\log \ell}$,
\begin{align*} \PP &\left(\langle X, \theta \rangle \geq c_1 \frac{ M}{\sqrt{n}} \cdot t \right)  = c_2 e^{-C_3 t^2} \frac{1}{c_2 e^{-C_3 t^2} \cdot \ell} \sum_{i=1}^\ell \PP\left(\langle X_i, \theta \rangle \geq c_1 \frac{ M}{\sqrt{n}} \cdot t \right) \\
& \geq c_2 e^{-C_3 t^2}  \cdot \PP \left( \# \left \{ i ; \langle X_i, \theta \rangle \geq c_1 \frac{ M}{\sqrt{n}} \cdot t  \right \} \geq c_2 e^{-C_3 t^2} \cdot \ell \right)
\geq c_2 e^{-C_3 t^2} \cdot (1 - C / \sqrt{\ell}),
\end{align*}
where the last passage is the content of Lemma \ref{lem_1114}. We may similarly obtain a corresponding
lower bound for $ \PP \left(\langle X, \theta \rangle \leq -c_1 t M / \sqrt{n} \right)$.
Since $\ell = \lfloor n^{1/8} \rfloor$,
the desired conclusion follows by adjusting the constants.
\end{proof}

\section{Geometry of the high-dimensional sphere}
\label{sec3}

This is the second section dedicated to the proof of Proposition \ref{prop_818}.
A few geometric properties of the high-dimensional sphere will be used here.
For example, the sphere $S^{n-1}$ does not contain more than $n$ mutually orthogonal vectors, yet it contains $e^{\eps n}$
mutually almost-orthogonal vectors. Moreover, for the purpose of computing the expectation of the supremum,
a family of $e^{\eps n}$ standard Gaussians which are almost-orthogonal in pairs behaves approximately like a collection of  independent Gaussians.

\medskip
While Corollary \ref{cor_103} takes care of the interval $t \in [0, 5 \sqrt{\log n}]$,
in this section we deal with the range  $t \in [5 \sqrt{\log n}, t_0]$
where $t_0$ is defined in (\ref{eq_954}). We begin with some background on Sudakov's minoration theorem and the concentration of measure inequality on the sphere.
Given a bounded, non-empty subset $S \subseteq \RR^n$, its {\it supporting functional} is defined via
$$ h_S(\theta) = \sup_{x \in S} \langle x, \theta \rangle \qquad \qquad (\theta \in \RR^n). $$
The supporting functional $h_S$ is a convex function on $\RR^n$ whose Lipschitz constant
is bounded by $R(S) = \sup_{x\in S} |x|$. The {\it mean width} of $S$ is $2 M^*(S)$ where
$$ M^*(S) = \int_{S^{n-1}} h_S(\theta) d \sigma_{n-1}(\theta). $$
The concentration inequality for Lipschitz functions on the sphere (see, e.g., \cite[Appendix V]{MS})
states that for any $r > 0$,
\begin{equation}
 \sigma_{n-1}
 \left( \left \{ v \in S^{n-1} \, ; \, |h_S(v) - M^*(S)| \geq r \cdot R(S) \right \} \right) \leq C e^{-c r^2 n}.
 \label{eq_715}
\end{equation}
A lower bound for $M^*(S)$ is provided by the following Sudakov's minoration theorem (see, e.g., \cite[Section 3.3]{LT}):

\begin{theorem}[Sudakov] Let $N \geq 1, \alpha > 0$ and let $x_1,\ldots,x_N \in \RR^n$.
Set $S = \{ x_1,\ldots, x_N \}$  and assume that $|x_i - x_j| \geq \alpha$
for any $i \neq j$. Then,
$$ M^*(S) \geq c \alpha \sqrt{ \frac{\log N}{n} }, $$
where $c > 0$ is a universal constant. \label{thm_721}
\end{theorem}

We shall need the following elementary lemma:

\begin{lemma} Let $Z_1,\ldots,Z_N$ be random variables attaining values in $\{ 0,1 \}$. Let $1 \leq k \leq N, 0 \leq \eps \leq 1$,
and assume that for any $A \subseteq \{ 1, \ldots, N \}$ with $\#(A) = k$,
\begin{equation}  \PP \left( \exists i \in A, \ Z_i = 1 \right) \geq 1 - \eps. \label{eq_427} \end{equation}
Then,
\begin{equation}  \PP \left( \sum_{i=1}^N Z_i \geq \frac{N}{3k} \right) \geq 1 - 2 \eps. \label{eq_1629}
\end{equation}
\label{lem_elem}
\end{lemma}

\begin{proof}
If $k \geq N/3$ then (\ref{eq_1629}) holds true, since it follows from (\ref{eq_427}) that with a probability of at least $1 - \eps$, there is a non-zero element among $Z_1,\ldots,Z_N$.
Suppose now that $k < N/3$.
The number of $k$-elements subsets $A \subseteq \{ 1,\ldots,N \}$
with $\max_{i \in A} Z_i = 0$ equals
$$ {N - \sum_{i=1}^N Z_i \choose k}. $$
Write $\cE$ for the event that $\sum_{i=1}^N Z_i \leq N / (3k)$. Conditioning on the event $\cE$,
$$ \frac{1}{{N \choose k}} \sum_{\#(A) = k} \PP \left( \forall i \in A, \ Z_i = 0 \, | \, \cE \right)
\geq \frac{{N - \lfloor N / (3k) \rfloor \choose k}}{{N \choose k} }
\geq \left(1 - \frac{N/(3k)}{N-k} \right)^k > \left(1 - \frac{1}{2k} \right)^k \geq \frac{1}{2}. $$
However, by (\ref{eq_427}),
$$ \eps \geq \frac{1}{{N \choose k}} \sum_{\#(A) = k} \PP \left( \forall i \in A, \ Z_i = 0 \right)
\geq  \frac{1}{{N \choose k}} \sum_{\#(A) = k} \PP(\cE)  \cdot \PP \left( \forall i \in A, \ Z_i = 0 \, | \, \cE \right)
\geq \PP(\cE) / 2.
$$
Hence $\PP(\cE) \leq 2 \eps$ and the lemma is proven.
\end{proof}

Sudakov's theorem is used in the following lemma:

\begin{lemma} Let $N \geq n$ and let $x_1,\ldots,x_N \in S^{n-1}$ be such that $\langle x_i, x_j \rangle \leq 49/50$ for any $i \neq j$.
Then there exists $\cF \subseteq S^{n-1}$ with $\sigma_{n-1}(\cF) \geq 1 - C / n^{c}$ such that for any $\theta \in \cF$,
\begin{equation}  \frac{\# \left \{ 1 \leq i \leq N \, ; \, \langle x_i, \theta \rangle \geq c_1 t / \sqrt{n} \right \}}{N} \geq c_2 e^{-C_3 t^2},
\qquad \text{for all} \  t \in [\sqrt{\log n} , \sqrt{\log N}],
\label{eq_734}
\end{equation}
where $c_1, c_2, C_3, c, C > 0$ are universal constants. \label{lem_506}
\end{lemma}

\begin{proof} Denote $S = \{ x_1,\ldots, x_N\} \subset S^{n-1}$ and
note that  $|x_i - x_j| \geq \sqrt{2 - 49/25} = 1/5$ for all $i \neq j$.
 Fix a number $t \in [\sqrt{\log n}, \sqrt{\log N}]$.
Let $A \subseteq  \{ x_1,\ldots, x_N \}$ be any subset with $\#(A) \geq \exp(t^2)$.
By Theorem \ref{thm_721},
\begin{equation} M^*(A) \geq c t / \sqrt{n}. \label{eq_901} \end{equation}
Next we will  apply the concentration inequality
(\ref{eq_715}) with $r = M^*(A) / (2 R(A))$.
Since $R(A) = 1$, it follows from  (\ref{eq_715}) and (\ref{eq_901}) that
$$
\sigma_{n-1}
 \left( \left \{ \theta \in S^{n-1} \, ; \, h_A(\theta) \geq M^*(A) / 2  \right \} \right)  \geq 1 - C \exp \left(-c n \left( \frac{M^*(A)}{R(A)} \right)^2  \right)
 \geq 1 - \tilde{C} e^{-\tilde{c} t^2}. $$
 Let $\Theta$ be a random vector, distributed uniformly over $S^{n-1}$.
By combining the last inequality with (\ref{eq_901}), we see that for any fixed subset $\tilde{A} \subseteq \{1, \ldots, N \}$
with $\#(\tilde{A}) = \lceil \exp(t^2) \rceil$,
$$ \PP \left( \exists i \in \tilde{A} \, ; \left \langle x_i, \Theta \right \rangle \geq c t / \sqrt{n} \right)
 \geq 1 - \tilde{C} e^{-\tilde{c} t^2}. $$
Let us now apply Lemma \ref{lem_elem} for $Z_i = 1_{\{ \langle x_i, \Theta \rangle \geq c t / \sqrt{n} \}}$.
Lemma \ref{lem_elem} now implies that with a probability of at least $1 - 2 \tilde{C} e^{-\tilde{c} t^2}$ of selecting $\Theta \in S^{n-1}$,
$$ \# \left \{ 1 \leq i \leq N \, ; \, \langle x_i, \Theta \rangle \geq c t / \sqrt{n} \right \} \geq \frac{N}{3 \lceil \exp(t^2) \rceil } \geq \frac{N}{6} \cdot e^{-t^2}. $$
We now let the parameter $t$ vary. Let $I$ be the collection of all integer powers of two that lie in the interval
$[\sqrt{\log n}, \sqrt{\log N}]$. Then,
$$ \PP \left( \forall t \in I,  \ \frac{\# \left \{ 1 \leq i \leq N \, ; \, \langle x_i, \Theta \rangle \geq c t / \sqrt{n} \right \}}{N} \geq \frac{e^{-t^2}}{6} \right)
\geq 1 - \sum_{t \in I} 2 \tilde{C} e^{-\tilde{c} t^2} \geq 1 - \frac{\hat{C}}{n^{\hat{c}}}. $$
The restriction $t \in I$ may be upgraded to the condition $t \in
[\sqrt{\log n}, \sqrt{\log N}]$ by adjusting the constants. The lemma is thus proven.
\end{proof}

Recall the construction of $\theta_1$ and $\theta_2$ from Section \ref{sec_stra},
and also the definition (\ref{eq_954}) of the parameter $t_0$. From the construction
we see that for any $v \in S^{n-1}$ with $|\langle v, \theta_1 \rangle| \leq 1/10$,
\begin{equation}
\PP \left( |X| \geq M \ \text{and} \ \left| \frac{X}{|X|} - v \right| \leq \frac{1}{5}  \right)
\leq 2 e^{-t_0^2}, \label{eq_147}
\end{equation}
where $M > 0$ satisfies $\PP( |X| \geq M ) \geq 1/3$ and $\PP( |X| \leq M ) \geq 2/3$.

\begin{lemma} Assume that $t_0 \geq 5 \sqrt{\log n}$ and set $N = \lfloor e^{t_0^2 / 4} \rfloor$.
Let $X_1,\ldots, X_N$ be independent copies of $X$. Then with a probability of at least $1 - C / n$
of selecting $X_1,\ldots,X_N$, there exists $I \subseteq \{ 1, \ldots, N \}$ with the following three properties:
\begin{enumerate}
\item[(i)] $\#(I) \geq N / 10$.
\item[(ii)] For any $i, j \in I$ with $i \neq j$ we have $\langle X_i, X_j \rangle \leq (49/50) \cdot |X_i| \cdot |X_j|$.
\item[(iii)] For  any $i \in I$,
$$ |X_{i}| \geq M, \quad |\langle X_{i}, \theta_1 \rangle| \leq 10 |X_{i}| /\sqrt{n} \quad \text{and}
\quad |\langle X_{i}, \theta_2 \rangle| \leq 10 |X_{i}|/\sqrt{n}. $$
\end{enumerate}
Here, $C  >0$ is a universal constant.
\label{lem_143}
\end{lemma}

\begin{proof} We may assume that $n \geq 10^4$, as otherwise for an appropriate choice of the constant $C$,
all we claim is that a certain event holds with a non-negative probability.
Write $$ \cA = \{ v \in \RR^n \, ; \, |v| \geq M, \, \max_{j=1,2} |\langle v/|v|, \theta_j \rangle| \leq 10 / \sqrt{n} \}. $$
According to Lemma \ref{lem_519}, for $i =1,\ldots,N$,
$$ \PP(X_i \in \cA) > 1/5. $$
Denote $I = \{ i =1,\ldots,N \, ; \, X_i \in \cA \}$. By the Chernoff large deviation bound,
$$ \PP( \#(I) \geq N / 10 ) \geq 1 - C \exp(-c N). $$
Note that  $10 / \sqrt{n} \leq 1/10$ and that  if $v \in \cA$ then $|\langle v/|v|, \theta_1 \rangle| \leq 1/10$.
It thus follows from (\ref{eq_147}) that for any $i,j \in \{ 1, \ldots, N \}$ with $i \neq j$,
\begin{align*}
 & \PP \left( i,j \in I \ \ \text{and} \ \ \left| \frac{X_j}{|X_j|} - \frac{X_i}{|X_i|} \right| \leq \frac{1}{5}  \right) \\
 & \leq  \PP \left(  X_j \in \cA \ \ \text{and} \ \ \left| \frac{X_j}{|X_j|} - \frac{X_i}{|X_i|} \right| \leq \frac{1}{5}  \ \ \Big| \ \ X_i \in \cA \right) \leq 2 e^{-t_0^2}
\leq \frac{2}{N^4}. \end{align*}
Consequently,
$$ \PP \left( \exists i, j \in I \ \text{with} \ i \neq j  \ \text{and}  \ \left| \frac{X_i}{|X_i|} - \frac{X_j}{|X_j|} \right| \leq \frac{1}{5} \right)
\leq \frac{N(N-1)}{2} \cdot \frac{2}{N^4} \leq \frac{1}{N^2}. $$
We conclude that with a probability of at least $1 - C \exp(-c N) - 1/N^2 \geq 1 - \tilde{C} / n$,
$$ \#(I) \geq N/10 \qquad \text{and} \qquad \forall i, j \in I, i \neq j \ \ \ \Longrightarrow \ \ \ \left| \frac{X_i}{|X_i|} - \frac{X_j}{|X_j|} \right| > \frac{1}{5}. $$
Note that $\langle X_i, X_j \rangle \leq (49/50) \cdot |X_i| \cdot |X_j|$ if and only if
$|X_i / |X_i| - X_j / |X_j|| \geq 1/5$. Thus conclusions (i), (ii) and (iii) hold true with a probability of at least $1 - \tilde{C} /n$, thereby completing the proof.
\end{proof}

By combining Lemma \ref{lem_506}
and Lemma \ref{lem_143} we arrive at the following:

\begin{lemma}
Assume that $t_0 \geq 5 \sqrt{\log n}$ and set $N = \lfloor e^{t_0^2 / 4} \rfloor$.
Then there exists a fixed subset $\cF \subseteq S^{n-1}$ with $\sigma_{n-1}(\cF) \geq 1 - C / n^c$
such that for any $\theta_3 \in \cF$ the following holds: Define $\theta$ via (\ref{eq_741}).
Let $X_1,\ldots,X_N$ be independent copies of the random vector $X$. Then with a probability of at least $1 - \tilde{C} / n^{\tilde{c}}$
of selecting $X_1,\ldots,X_N$,
\begin{equation}   \frac{\# \left \{ 1 \leq i \leq N \, ; \, \langle X_i, \theta \rangle \geq c_1 \frac{M}{\sqrt{n}} \cdot  t  \right \}}{N} \geq c_2 e^{-C_3 t^2},
\qquad \text{for all} \  t \in [\sqrt{\log n},  t_0 ],
\label{eq_1713} \end{equation}
and
\begin{equation}   \frac{\# \left \{ 1 \leq i \leq N \, ; \, \langle X_i, \theta \rangle \leq -c_1 \frac{M}{\sqrt{n}} \cdot  t  \right \}}{N} \geq c_2 e^{-C_3 t^2},
\qquad \text{for all} \  t \in [\sqrt{\log n},  t_0 ].
\label{eq_1714} \end{equation}
Here, $c_1, c_2, C_3, c, C, \tilde{c}, \tilde{C} > 0$ are universal constants.
\label{lem_1715}
\end{lemma}

\begin{proof} This proof is almost identical to the deduction
of Lemma \ref{lem_1114} from Lemma \ref{lem_1013} and Lemma \ref{lem_1009}.
Let us spell  out the details. Set $\vec{X} = (X_1,\ldots,X_N)$
and let $\Theta$ be a random vector, independent of $\vec{X}$, distributed uniformly on $S^{n-1}$.
We say that $\vec{X} \in \cA_1$ if the event described in Lemma \ref{lem_143} holds true.
Thus, $$ \PP(\vec{X} \in \cA_1) \geq 1 - C /  n. $$
Assuming that $\vec{X} \in \cA_1$, we may apply Lemma \ref{lem_506}
and obtain that with a probability of at least $1 - \tilde{C} / n^{\tilde{c}}$ of selecting $\Theta \in S^{n-1}$,
$$ \# \left \{ 1 \leq i \leq N \, ; \, \left \langle \frac{X_i}{|X_i|}, \Theta \right \rangle \geq c_1 t / \sqrt{n} \right \} \geq c_2 e^{-C_3 t^2} \cdot (N/10)
\quad \text{for all} \  t \in [\sqrt{\log n} , \sqrt{\log N}]. $$
Assuming that $\vec{X} \in \cA_1$, we may use Lemma \ref{lem_143}(iii) in order to conclude that
with a probability of at least $1 - \tilde{C} / n^{\tilde{c}}$ of selecting $\Theta \in S^{n-1}$,
for $ t \in [\sqrt{\log n}, 4 \sqrt{\log N}]$,
\begin{equation} \# \left \{ 1 \leq i \leq N \, ; \, \left \langle X_i, \frac{\theta_1 - \theta_2 + \Theta}{|\theta_1 - \theta_2 + \Theta|} \right \rangle \geq
\bar{c}_1 \frac{M}{\sqrt{n}} \cdot  t  \right \} \geq \bar{c}_2 e^{-\bar{C}_3 t^2} \cdot N.
 \label{eq_1738} \end{equation}
Write $\cA_2$ for the event that (\ref{eq_1738}) holds true for all $ t \in [\sqrt{\log n}, 4 \sqrt{\log N}]$. Thus, $$ \PP( (\Theta, \vec{X}) \in \cA_2 ) \geq 1 - C /  n - \tilde{C} / n^{\tilde{c}} \geq 1 - \bar{C} / n^{\bar{c}}. $$
Consequently, there exists $\cF \subseteq S^{n-1}$ with
$$ \sigma_{n-1}(\cF) \geq 1 - \hat{C} / n^{\hat{c}} $$
with the following property: For any $\theta_3 \in \cF$,
with a probability of at least $1 - \hat{C} / n^{\hat{c}}$ of selecting $X_1,\ldots,X_{N}$,
for all $t \in [\sqrt{\log n}, 4 \sqrt{\log N}]$,
$$ \# \left \{ 1 \leq i \leq N \, ; \, \left \langle X_i, \frac{\theta_1 - \theta_2 + \theta_3}{|\theta_1 - \theta_2 + \theta_3|} \right \rangle \geq c_1 \frac{M}{\sqrt{n}} \cdot  t  \right \} \geq c_2 e^{-C_3 t^2} \cdot N. $$
Recalling that $4 \sqrt{\log N} \geq t_0$, we have established (\ref{eq_1713}).
The proof of (\ref{eq_1714}) is similar.
\end{proof}

The short proof of the following corollary is analogous to that of Corollary \ref{cor_103}.

\begin{corollary}
There exists a fixed subset $\cF \subseteq S^{n-1}$ with $\sigma_{n-1}(\cF) \geq 1 - C / n^{c}$
such that for any $\theta_3 \in \cF$, defining $\theta$ via (\ref{eq_741}),
$$ \forall t \in [\sqrt{\log n}, t_0], \quad \min \left \{ \PP \left( \langle X, \theta \rangle \geq c_1 \frac{M}{\sqrt{n}} \cdot t \right),
\PP \left( \langle X, \theta \rangle \leq -c_1 \frac{M}{\sqrt{n}} \cdot t \right) \right \}  \geq c_2 e^{-C_3 t^2}, $$
where $c, C, c_1, c_2, C_3 > 0$ are universal constants.
\label{cor_1726}
\end{corollary}

\begin{proof} We may assume that $n$ exceeds a certain fixed universal constant.
Let $\cF$ be the set from Lemma \ref{lem_1715}, denote $N = \lfloor \exp(t_0^2/4)\rfloor$,
and let $X_1,\ldots,X_N$ be independent copies of $X$. Then for any $\theta_3 \in \cF$, defining $\theta$ via (\ref{eq_741})
we have that for any $t \in [\sqrt{\log n}, t_0]$,
\begin{align*} \PP &\left(\langle X, \theta \rangle \geq c_1 \frac{ M}{\sqrt{n}} \cdot t \right)
\geq c_2 e^{-C_3 t^2}  \cdot \PP \left( \frac{\# \left \{ i ; \langle X_i, \theta \rangle \geq c_1 \frac{ M}{\sqrt{n}} \cdot t  \right \}}{N} \geq c_2 e^{-C_3 t^2}  \right)
\geq \frac{c_2}{2} e^{-C_3 t^2},
\end{align*}
where the last passage is the content of Lemma \ref{lem_1715}. The bound for $ \PP \left(\langle X, \theta \rangle \leq -c_1 t M / \sqrt{n} \right)$
is proven similarly.
\end{proof}

\section{Proof of the main proposition}
\label{sec4}

In this section we complete the proof of Proposition \ref{prop_818}.
We begin with the following standard  observation:

\begin{lemma} Suppose that $X$ is a random vector in $\RR^n$
with $\PP(X = 0) = 0$. Then there exists a fixed subset $\cF \subseteq S^{n-1}$ of full measure,
such that $\PP( \langle X, \theta \rangle  = 0) = 0$ for all $\theta \in \cF$. \label{lem_1127}
\end{lemma}

\begin{proof} For $a > 0$, we say that a subspace $E \subseteq \RR^n$ is $a$-basic if $\PP(X \in E) \geq a$
while $\PP(X \in F) < a$ for all subspaces $F \subsetneq E$. Lemma 7.1 in \cite{K_euro} states that
there are only finitely many subspaces that are $a$-basic for any fixed $a > 0$. Write $\cS$ for the collection
of all subspaces that are $a$-basic for some rational number $a > 0$. Then $\cS$ is a countable family
which does not contain the subspace $\{ 0 \}$. Consequently, the set
$$ \cF = \{ \theta \in S^{n-1} \, ; \, \forall E \in \cS, \ E \not \subset \theta^{\perp} \} $$
is a set of full measure in $S^{n-1}$, as its complement is the countable union of spheres of lower dimension. Here, $\theta^{\perp} = \{ x \in \RR^n \, ; \, \langle x, \theta \rangle = 0 \}$.
Suppose that $\theta \in \cF$, and let us prove that $\PP( \langle X, \theta \rangle  = 0) = 0$. Otherwise, there exists a rational number $a > 0$ such that
$$ \PP( \langle X, \theta \rangle  = 0) \geq a. $$
Thus $\theta^{\perp}$ contains an $a$-basic subspace, contradicting the definition of $\cF$.
\end{proof}

Recall the definition of $M, \theta_1$ and $\theta_2$ from Section \ref{sec_stra}.

\begin{lemma} Let $\cF_3 \subseteq \left \{ \theta_3 \in S^{n-1} \, ; \,
|\langle \theta_3, \theta_1 \rangle| \leq \frac{1}{10}  \ \text{and} \ |\langle \theta_3, \theta_2 \rangle| \leq \frac{1}{10} \right \}$.
Then for any $\theta_3 \in \cF_3$ and $v \in S^{n-1}$,
\begin{equation}  |v - \theta_1| \leq \frac{1}{5} \qquad \Longrightarrow \qquad \langle v, \theta_1 - \theta_2 + \theta_3  \rangle \geq \frac{1}{10}, \label{eq_1408}
\end{equation}
and
\begin{equation} |v - \theta_2| \leq \frac{1}{5} \qquad \Longrightarrow \qquad \langle v, \theta_1 - \theta_2 + \theta_3  \rangle \leq -\frac{1}{10}. \label{eq_1410}
\end{equation} \label{lem_1421}
\end{lemma}

\begin{proof} Recall that $\left| \langle \theta_1, \theta_2 \rangle \right| \leq 1/10$.
Note that for any $\theta_3 \in \cF_3$ and $i,j \in \{ 1, 2, 3 \}$ with $i \neq j$,
$$ \sqrt{9/5} \leq |\theta_i - \theta_j| \leq \sqrt{11/5}. $$
Let $v \in S^{n-1}$
be any vector with $|v - \theta_1| \leq 1/5$. Then for any $\theta_3 \in \cF_3$ and $j=2,3$
we have that $$
\sqrt{\frac{9}{5}} - \frac{1}{5} \leq |\theta_j - \theta_1| - |\theta_1 - v|
\leq |v - \theta_j| \leq |\theta_j - \theta_1| + |\theta_1 - v| \leq \sqrt{\frac{11}{5}} + \frac{1}{5},  $$
and hence for $j=2,3$,
\begin{equation}  \langle v, \theta_j \rangle = 1 - \frac{1}{2} \cdot |v - \theta_j|^2 \in
\left[1 - \frac{1}{2} \cdot \left( \sqrt{\frac{11}{5}} + \frac{1}{5}   \right)^2, 1  - \frac{1}{2} \cdot \left( \sqrt{\frac{9}{5}} - \frac{1}{5}   \right)^2 \right] \subseteq
\left[ -\frac{3}{7}, \frac{3}{7} \right]. \label{eq_1409} \end{equation}
However, $\langle v, \theta_1 \rangle \geq 49 / 50$ for such $v$, and hence (\ref{eq_1408}) follows from (\ref{eq_1409}).
By replacing the triplet $(\theta_1, \theta_2, \theta_3)$ by $(\theta_2, \theta_1, -\theta_3)$ and repeating the above argument, we obtain (\ref{eq_1410}).
\end{proof}

\begin{proof}[Proof of Proposition \ref{prop_818}]
From Corollary \ref{cor_103} and Corollary \ref{cor_1726} we learn that there exists $\cF \subseteq S^{n-1}$ with
$\sigma_{n-1}(\cF_3) \geq 1 - C / n^{c}$
such that for any $\theta_3 \in \cF$, defining $\theta$ via (\ref{eq_741}),
\begin{equation} \forall t \in [0, t_0], \ \ \min \left \{ \PP \left( \langle X, \theta \rangle \geq c_1 \frac{M}{\sqrt{n}} \cdot t \right),
\PP \left( \langle X, \theta \rangle \leq -c_1 \frac{M}{\sqrt{n}} \cdot t \right) \right \}  \geq c_2 e^{-C_3 t^2}. \label{eq_1125}
\end{equation}
According to Lemma \ref{lem_1127}, we may remove a set of measure zero from $\cF$ and  additionally assume
that $\PP( \langle X, \theta \rangle = 0) = 0$. From Lemma \ref{lem_746} we learn  that any median of $|\langle X, \theta \rangle|$ is at most $C M / \sqrt{n}$.
Hence (\ref{eq_1125}) shows that for any $\theta_3 \in \cF$, defining $\theta$ via (\ref{eq_741}) we have that $\langle X, \theta \rangle$
is Super-Gaussian of length $c_1 t_0$, with parameters $c_2, c_3 > 0$. We still need to increase the length to $c_1 \sqrt{n}$.
To this end, denote
$$ \cF_3 = \left \{ \theta_3 \in \cF \ ; \ |\langle \theta_3, \theta_1 \rangle| \leq \frac{1}{10}  \quad \text{and} \quad
|\langle \theta_3, \theta_2 \rangle| \leq \frac{1}{10} \right \}. $$
Then $\sigma_{n-1}(\cF_3) \geq \sigma_{n-1}(\cF) - C \exp(-c n) \geq 1 - \tilde{C} / n^{\tilde{c}}$.
Recall from Section \ref{sec_stra} that for $j=1,2$,
\begin{equation}
\PP \left( |X| \geq M \ \text{and} \ \left| \frac{X}{|X|} - \theta_j \right| \leq \frac{1}{5}  \right)
\geq \frac{1}{2} \cdot e^{-t_0^2}. \label{eq_1423}
\end{equation}
Let us fix $t \in [t_0, \sqrt{n}], \theta_3 \in \cF_3$ and define $\theta$ via (\ref{eq_741}).
Since $0 < |\theta_1 - \theta_2 + \theta_3| \leq 3$, by (\ref{eq_1423}) and Lemma \ref{lem_1421},
\begin{align*}  \PP   \left( \langle X, \theta \rangle \geq \frac{M t}{30 \sqrt{n}}  \right)  & \geq
\PP \left( \langle X, \theta_1 - \theta_2 + \theta_3 \rangle \geq \frac{M t}{10 \sqrt{n}} \right)
 \geq \PP \left( \left \langle \frac{X}{|X|}, \theta_1 - \theta_2 + \theta_3 \right \rangle \geq \frac{M}{10 |X|} \right)
\\ & \geq
\PP \left( |X| \geq M, \ \left |  \frac{X}{|X|} - \theta_1 \right| \leq \frac{1}{5} \right) \geq \frac{1}{2} \cdot e^{-t_0^2} \geq \frac{1}{2} \cdot e^{-t^2}.
\end{align*}
Similarly,
\begin{align*}  \PP   \left( \langle X, \theta \rangle \leq -\frac{M t}{30 \sqrt{n}}  \right)  & \geq
 \PP \left( \left \langle \frac{X}{|X|}, \theta_1 - \theta_2 + \theta_3 \right \rangle \leq -\frac{M}{10 |X|} \right)
\\ & \geq
\PP \left( |X| \geq M, \ \left |  \frac{X}{|X|} - \theta_2 \right| \leq \frac{1}{5} \right) = e^{-t_0^2} \geq  e^{-t^2}.
\end{align*}
Therefore, we may upgrade (\ref{eq_1125}) to the following statement: For any $\theta_3 \in \cF$ and $t \in [0, \sqrt{n}]$, defining
$\theta$ via (\ref{eq_741}),
$$ \min \left \{ \PP \left( \langle X, \theta \rangle \geq c_1 \frac{M}{\sqrt{n}} \cdot t \right),
\PP \left( \langle X, \theta \rangle \leq -\hat{c}_1 \frac{M}{\sqrt{n}} \cdot t \right) \right \}  \geq \hat{c}_2 e^{-\hat{C}_3 t^2}. $$
We have thus proven that $\langle X, \theta \rangle$ is Super-Gaussian of length $\overline{c}_1 \sqrt{n}$ with parameters $\overline{c}_2, \overline{c}_3 > 0$.
\end{proof}

\section{Angularly-isotropic position}
\label{sec_ang}

In this section we deduce Theorem \ref{thm_1140} from Proposition \ref{prop_818}
by using the angularly-isotropic position which is discussed below. We begin with the following:

\begin{lemma} Let $d, X, \cB$ be as in Theorem \ref{thm_1140}. Set $n = \lceil d \rceil$.
Then there exists a fixed linear map $T: \cB \rightarrow \RR^{n}$
such that for any $\eps > 0$, the random vector $T(X)$ is of class $\effrank_{\geq d - \eps}$. \label{lem_1117}
\end{lemma}

\begin{proof} We will show that a generic linear map $T$ works. Denote $N = \dim(\cB)$  and identify $\cB \cong \RR^N$.
Since the effective rank of $X$ is at least $d$, necessarily $d \leq N$ and hence also $n = \lceil d \rceil \leq N$.
Let $L \subseteq \RR^N$ be a random $n$-dimensional subspace, distributed uniformly in the Grassmannian $G_{N,n}$.
Denote $T = \proj_L: \RR^N \rightarrow L$, the orthogonal projection operator onto the subspace $L$. 

\medskip For any fixed subspace $E \subseteq \RR^N$, with probability one of selecting  $L \in G_{N,n}$,
$$ \dim(\ker(T) \cap E) = \max \{ 0, \dim(E) - n \}, $$
or equivalently,
\begin{equation}
 \dim(T(E)) = \dim(E) - \dim(\ker(T) \cap E) = \min \{ n, \dim(E) \}. \label{eq_1000}
 \end{equation}
Recall that for $a > 0$, a subspace $E \subseteq \RR^N$ is $a$-basic if $\PP(X \in E) \geq a$
while $\PP(X \in F) < a$ for all subspaces $F \subsetneq E$. Lemma 7.1 in \cite{K_euro} states that
there exist only countably  many subspaces that are $a$-basic with $a$ being a positive, rational number.
Write $\cG$ for the collection of all these basic subspaces. Then with probability one of selecting $L \in G_{N, n}$,
\begin{equation}
 \forall E \in \cG, \qquad \dim(T(E)) = \min \{ n, \dim(E) \}. \label{eq_919}
 \end{equation}
We now fix a subspace $L \in G_{N,n}$ for which $T = \proj_L$ satisfies (\ref{eq_919}). Let $S \subseteq L$ be any subspace and assume that $a \in \QQ \cap (0,1]$  satisfies
$$ \PP(T(X) \in S) \geq a. $$
Then $\PP(X \in T^{-1}(S)) \geq a$. Therefore $T^{-1}(S)$ contains an $a$-basic subspace $E$. Thus
$E \in \cG$ while  $E \subseteq T^{-1}(S)$ and $\PP(X \in E) \geq a$. Since the effective rank of $X$ is at least $d$, necessarily $\dim(E) \geq a \cdot d$.
Since $T(E) \subseteq S$, from (\ref{eq_919}),
$$ \dim(S) \geq \dim(T(E)) = \min \{ n, \dim(E) \} \geq 
\min \{ n, \lceil a \cdot d \rceil \} = \lceil a \cdot d \rceil. $$
We have thus proven that for any subspace $S \subseteq L$ and $a \in \QQ \cap (0, 1]$,
\begin{equation} \PP(T(X) \in S) \geq a  \qquad \Longrightarrow \qquad \dim(S) \geq \lceil a \cdot d \rceil. \label{eq_1034} \end{equation}
It follows from (\ref{eq_1034}) that for any subspace $S \subseteq L$,
$$ \PP(T(X) \in S) \leq  \dim(S) / d. $$
This implies that for any $\eps > 0$, the random vector $T(X)$ is of class $\effrank_{\geq d - \eps}$ .
\end{proof}

\begin{lemma} Let $d, X, \cB$ be as in Theorem \ref{thm_1140}.
Assume that $d < \dim(\cB)$ and that for any subspace $\{ 0 \} \neq E \subsetneq \cB$,
\begin{equation} \PP(X \in E) < \dim(E) / d. \label{eq_1038} \end{equation}
Then there exists $\eps  > 0$ such that $X$ is of class $\effrank_{\geq d + \eps}$.  \label{lem_1053}
\end{lemma}

\begin{proof} Since the effective rank of $X$ is at least $d$, necessarily $\PP(X = 0) = 0$.
Assume by contradiction that for any $\eps > 0$, the random vector $X$ is not of class $\effrank_{\geq d + \eps}$. 
Then for  any $\eps > 0$ there exists a subspace $\{0 \} \neq E \subseteq \cB$ with
$$ \PP(X \in E) \geq - \eps + \dim(E) / d . $$
The Grassmannian of all $k$-dimensional subspaces of $\cB$ is compact. Hence there is a dimension $1 \leq k \leq \dim(\cB)$ and a converging sequence of $k$-dimensional subspaces $E_1,E_2,\ldots \subseteq \cB$
with
\begin{equation} \PP(X \in E_{\ell}) \geq -1/\ell + \dim(E_{\ell}) / d = - 1/\ell + k/d  \qquad \qquad \text{for all} \ \ell \geq 1. \label{eq_1804} \end{equation}
Denote $E_0 = \lim_{\ell} E_{\ell}$, which is a $k$-dimensional subspace in $\cB$.
Let $U \subseteq \cB$ be an open neighborhood of $E_0$ with the property that $tx \in U$ for all $x\in U, t \in \RR$.
Then $E_{\ell} \subseteq U$ for a sufficiently large $\ell$, and we learn from (\ref{eq_1804}) that
\begin{equation} \PP(X \in U) \geq k/d.
\label{eq_1759}
\end{equation}
Since $E_0$ is the intersection of a decreasing sequence of such neighborhoods $U$, it follows from (\ref{eq_1759}) that
\begin{equation} \PP(X \in E_0) \geq k / d = \dim(E_0) / d.
\label{eq_1039}
\end{equation}
Since $d < \dim(\cB)$, the inequality in (\ref{eq_1039}) shows that $E_0 \neq \cB$. Hence $1 \leq \dim(E_0) \leq \dim(\cB)-1$, and (\ref{eq_1039}) contradicts (\ref{eq_1038}).
The lemma is thus proven.
\end{proof}

The following lemma is a variant of Lemma 5.4 from \cite{K_euro}.

\begin{lemma} Let $d, X, \cB$ be as in Theorem \ref{thm_1140}.
Then there exists a fixed scalar product $\langle \cdot, \cdot \rangle$ on $\cB$ such that denoting $| \theta | = \sqrt{\langle \theta, \theta \rangle}$,
we have
\begin{equation}  \EE \left \langle \frac{X}{| X |}, \theta \right \rangle^2 \leq \frac{| \theta |^2}{d} \qquad \qquad \text{for all} \ \theta \in \cB. \label{eq_1030}
\end{equation}
\label{lem_1010}
\end{lemma}

\begin{proof} By induction on the dimension $n = \dim(\cB)$.
Assume first that there exists a subspace $\{ 0 \} \neq E \subsetneq \cB$,
such that equality holds true in (\ref{eq_1011}). In this case, there exists a subspace $F \subseteq \cB$
with $E \oplus F = \cB$ and $\PP(X \in E \cup F) = 1$. We will construct a scalar product in $\cB$ as follows:
Declare that $E$ and $F$ are orthogonal subspaces, and use the induction
hypothesis in order to find appropriate scalar products in the subspace $E$ and in the subspace $F$.
This induces a scalar product in $\cB$ which satisfies
$$ \EE \left \langle \frac{X}{|X|}, \theta \right \rangle^2 \leq \frac{| \theta |^2}{d} \qquad \qquad \text{for all} \ \theta \in E \cup F. $$
For any $\theta \in \cB$ we may  decompose $\theta = \theta_E + \theta_F$ with $\theta_E \in E, \theta_F \in F$.
Since $\PP(X \in E \cup F) = 1$, we obtain
$$ \EE \left \langle \frac{X}{|X|}, \theta \right \rangle^2  =
\EE \left \langle \frac{X}{|X|}, \theta_E \right \rangle^2
+ \EE \left \langle \frac{X}{|X|}, \theta_F \right \rangle^2 \leq \frac{|\theta_E|^2 + |\theta_F|^2}{d} = \frac{|\theta|^2}{d}, $$
proving (\ref{eq_1030}).

\medskip 
Next, assume that for any subspace $\{ 0 \} \neq E \subsetneq \cB$, the inequality in (\ref{eq_1011}) is strict.
There are two distinct cases, either $d = n$ or $d < n$. Consider first the case where $d = n = \dim(\cB)$. Thus,
for any subspace $E \subseteq \cB$ with $E \neq \{ 0 \}$ and $E \neq \cB$,
$$
\PP(X \in E) < \dim(E) / n. $$
This  is precisely the main assumption of Corollary 5.3 in \cite{K_euro}. By the conclusion 
of the corollary, there exists a scalar product in $\cB$ such that (\ref{eq_1030}) holds true. 
We move on to the case where $d < n$. Here,  we apply Lemma \ref{lem_1053} and conclude that $X$ is of class $\effrank_{\geq d + \eps}$ for some $\eps > 0$.
Therefore, for some $\eps > 0$, 
\begin{equation}
\PP(X \in E) < \dim(E) / (d + \eps) \qquad \qquad \forall E \subseteq \cB. \label{eq_955}
\end{equation}
Now we invoke Lemma 5.4 from \cite{K_euro}. Its assumptions are satisfies thanks to (\ref{eq_955}). 
From the conclusion of that lemma, there exists a scalar product in $\cB$ for which (\ref{eq_1030}) holds true.
\end{proof}

The condition that the effective rank of $X$ is at least $d$ is not only sufficient but  is 
also necessary for the validity of conclusion (\ref{eq_1030}) from Lemma \ref{lem_1010}. Indeed, it follows from (\ref{eq_1030}) 
that for any subspace $E \subseteq \cB$,
\begin{equation}  \PP(X \in E) \leq \EE \left| \proj_E \frac{X}{|X|} \right|^2 = \sum_{i=1}^{\dim(E)} \EE \left \langle \frac{X}{|X|}, u_i \right \rangle^2 \leq \frac{\dim(E)}{d}, 
\label{eq_1005} \end{equation}
where $u_1,\ldots,u_m$ is an orthonormal basis of the subspace $E$ with $m = \dim(E)$. Equality in (\ref{eq_1005}) holds true 
if and only if $\PP(X \in E \cup E^{\perp}) = 1$, where $E^{\perp}$ is the orthogonal complement to $E$. Consequently, the effective rank 
of  $X$ is at least $d$.

\begin{definition} Let $X$ be a random vector in $\RR^n$ with $\PP(X = 0) = 0$.
We say that $X$ is angularly-isotropic if
\begin{equation}
 \EE \left \langle \frac{X}{|X|}, \theta \right \rangle^2  = \frac{1}{n} \qquad \qquad \text{for all} \ \theta \in S^{n-1}.
 \label{eq_810}
 \end{equation}
For $0 < d \leq n$ we say that $X/|X|$ is sub-isotropic with parameter $d$ if
\begin{equation} \EE \left \langle \frac{X}{|X|}, \theta \right \rangle^2  \leq \frac{1}{d} \qquad \qquad \text{for all} \ \theta \in S^{n-1}.
\label{eq_1620} \end{equation}
\end{definition}

We observe that $X$ is angularly-isotropic if and only if $X/|X|$ is sub-isotropic with parameter $n$.
Indeed, suppose that (\ref{eq_1620}) holds true with $d = n$. Given any $\theta \in S^{n-1}$ we may find an orthonormal basis
$\theta_1,\ldots, \theta_n \in \RR^n$ with $\theta_1 =\theta$. Hence
$$ 1 = \EE \left| \frac{X}{|X|} \right|^2 = \EE \sum_{i=1}^n \left \langle \frac{X}{| X |}, \theta_i \right \rangle^2 \leq \sum_{i=1}^n \frac{1}{n} = 1, $$
and (\ref{eq_810}) is proven. 

\begin{proof}[Proof of Theorem \ref{thm_1140}]
According to Lemma \ref{lem_1117}, we may project $X$ to a lower-dimensional space,
and assume that $\dim(\cB) = n = \lceil d \rceil$ and that the effective rank of $X$ is at least $n/2$.
Lemma \ref{lem_1010} now shows that there exists a scalar product in $\cB$ with respect to which
$X/|X|$ is sub-isotropic with parameter $n/2$. We may therefore identify $\cB$ with $\RR^n$ so that
$$ \EE \left \langle \frac{X}{|X|}, \theta \right \rangle^2  \leq \frac{2}{n} \qquad \qquad \text{for all} \ \theta \in S^{n-1}.
$$
Thus condition (\ref{eq_1051}) of Proposition \ref{prop_818} is verified. By the conclusion of Proposition \ref{prop_818},
there exists a non-zero linear functional $\ell: \RR^n \rightarrow \RR$ such that $\ell(X)$ is Super-Gaussian of length $c_1 \sqrt{n} \geq c \sqrt{d}$
with parameters $c_2, c_3 > 0$.
\end{proof}

\begin{proof}[Proof of Corollary \ref{cor_1405}]
By assumption, 
 $\PP(X \in E) \leq \dim(E)  / d$
for any finite-dimensional subspace $E \subseteq \cB$.
Lemma 7.2 from \cite{K_euro} states that there exists a continuous, linear map $T: \cB \rightarrow \RR^N$
such that $T(X)$ has an effective rank of at least $d/2$. We may now invoke Theorem \ref{thm_1140} for the random vector $T(X)$, and 
conclude that for some non-zero, fixed, linear
functional $\ell: \RR^N \rightarrow \RR$, the random variable $(\ell \circ T)(X)$ is Super-Gaussian of length $c_1 \sqrt{d}$ with parameters $c_2, c_3 > 0$.
\end{proof}

\begin{remark} {\rm 
We were asked by Yaron Oz about analogs of Theorem \ref{thm_1030} 
in the hyperbolic space. We shall work with the standard hyperboloid model
$$ \HH^n = \left \{ (x_0,\ldots,x_n) \in \RR^{n+1} \, ; \, -x_0^2 + \sum_{i=1}^n x_i^2 = - 1, x_0 > 0 \right \} $$
where the Riemannian metric tensor is $g = -dx_0^2 + \sum_{i=1}^n d x_i^2$. 
For any linear subspace  $L \subseteq \RR^{n+1}$, the intersection $L \cap \HH^n$ is 
a totally-geodesic submanifold of $\HH^n$ which is called a hyperbolic subspace. 
When we discuss the dimension of a hyperbolic subspace, we refer to its dimension as a smooth manifold. Note 
that an $(n-1)$-dimensional hyperbolic subspace $E \subseteq \HH^n$ divides $\HH^n$ into two sides. A signed distance function 
$d_E: \HH^n \rightarrow \RR$ is a function that equals the hyperbolic distance to $E$ on one of these sides, and minus 
the distance to $E$ on the other side. Given 
a linear functional $\ell: \RR^{n+1} \rightarrow \RR$  such that 
$E = \HH^n \cap \{ x \in \RR^{n+1} \, ; \, \ell(x) = 0 \}$ we may write 
$$ d_E (x) = {\rm arcsinh}( \alpha \cdot \ell(x) ) \qquad \qquad (x \in \HH^n) $$
for some $0 \neq \alpha \in \RR$. It follows from Theorem \ref{thm_1140} that for any absolutely-continuous random vector $X$ in $\HH^n$, there exists an $(n-1)$-dimensional hyperbolic subspace
$E \subseteq \HH^n$ and an associated signed distance function $d_E$ such that the random variable $\sinh(d_E(X))$ is
Super-Gaussian of length $c_1 \sqrt{n}$ with parameters $c_2, c_3 > 0$. 
In general, we cannot replace the random variable $\sinh(d_E(X))$ in the preceding statement by  $d_E(X)$ itself. This 
is witnessed by the example of the random vector 
$$ X = \left(\sqrt{1 + R^2 \sum_{i=1}^n Z_i^2}, R Z_1,\ldots,R Z_n \right) \in \RR^{n+1} $$
which is supported in $\HH^n$. Here, $Z_1,\ldots,Z_n$ are independent standard Gaussian random variables,
and $R > 1$ is a fixed, large parameter.
}\end{remark}
{
}

\bigskip
\noindent School of Mathematical Sciences, Tel Aviv University, Tel Aviv 69978, Israel. \\ \noindent E-mail: klartagb@tau.ac.il

\end{document}